\documentclass[pdflatex,sn-mathphys-num]{sn-jnl}
\usepackage{graphicx}%
\usepackage{multirow}%
\usepackage{amsmath,amssymb,amsfonts}%
\usepackage{amsthm}%
\usepackage{mathrsfs}%
\usepackage[title]{appendix}%
\usepackage{xcolor}%
\usepackage{textcomp}%
\usepackage{manyfoot}%
\usepackage{booktabs}%
\usepackage{algorithm}%
\usepackage{algorithmicx}%
\usepackage{algpseudocode}%
\usepackage{listings}%
\newtheorem{lemma}{Lemma}
\newtheorem{corollary}{Corollary}

\theoremstyle{thmstyleone}%
\newtheorem{theorem}{Theorem}
\theoremstyle{thmstyletwo}%
\newtheorem{remark}{Remark}%
\theoremstyle{thmstylethree}%
\begin{document}

\title[BickleyNaylor Functions]{A Finitely Generated Module Representation of the Bickley-Naylor Functions}


\author[1,2]{\fnm{Anthony A. Ruffa}}
\email{ruffa1498@gmail.com}

\author*[2,3]{\fnm{Bourama Toni}}
\email{bourama.toni@howard.edu}

\equalcont{These authors contributed equally to this work.}


\affil[1]{\orgdiv{Innovative Integrals}, \orgname{LLC}, \orgaddress{\street{50 Cedarwood Lane}, \city{Hope Valley},  \state{RI}, \postcode{02832}, \country{USA}}}

\affil*[2]{\orgdiv{Department of Mathematics}, \orgname{Howard University}, \orgaddress{\street{2441 Sixth Street NW}, \city{Washington}, \state{DC}, \postcode{20059},  \country{USA}}}



\abstract{We present a novel and non-standard derivation of the Bickley-Naylor functions $Ki_n, n\in\mathbb{N}_0.$ A $3\times 3$ system of equations is developed, the solution of which yields new explicit expressions for the Bickley-Naylor functions $Ki_2$, $Ki_3$, and $Ki_4$. Using a recurrence relation, expressions follow for any other order of Bickley-Naylor functions. We then characterize the infinite set of $Ki_n, n\in\mathbb{N}_0$ as a four-dimensional span of modified Bessel and Struve functions over the ring of real polynomials, thus providing a new computational and modeling tool for studying and applying the Bickley-Naylor functions.}




\maketitle

\section{Introduction}

The Bickley-Naylor functions \cite{bickley1935} can be useful for addressing problems involving radiative heat transfer \cite{Altaç1996, Wang, Djeumegni} and neutron transport \cite{lewismiller}. They cannot be evaluated analytically; they are typically tabulated (as in \cite{bickley1935}), approximated, or evaluated numerically, prompting the need for a variety of approaches. Here we introduce a non-standard approach within the framework of exponential kernel transforms having the form $F(s)=\int_a^b f(x) \:e^{-s\psi(x)}\: dx$ that are extensively exposed in \cite{Ruffa2022,Ruffa2024}.

Bickley-Naylor functions are given by various equivalent expressions, e.g., 
$\forall n\in\mathbb{N}_0$

\begin{equation}
Ki_n(x)
= \int_0^{\frac{\pi}{2}} e^{-x\sec\theta} \cos^{n-1}\theta \:d\theta\\
= \int_0^{\infty} e^{-x\cosh t} \cosh^{-n} t \:dt.
\end{equation}
They are also defined using the repeated integrals
\begin{equation}
    Ki_n(x)=\int_x^\infty Ki_{n-1}(t) dt,
\end{equation}
where $Ki_0(x) = K_0(x).$ They have the recursive properties \cite{Abramowitz}

\begin{equation}
\begin{split}
Ki_0(x)= & K_0(x);\\
n\,Ki_{n+1}(x) = & -x\,Ki_n(x) + (n-1)\,Ki_{n-1}(x) + x\,Ki_{n-2}(x); n\ge 2,
\label{RR}
\end{split}
\end{equation} 
which supports the derivation of the $n$th order Bickley-Naylor function $Ki_n$ in terms of the zero-order modified Bessel function $K_0$, i.e.,

\begin{equation}
    Ki_n(x) = \frac{1}{(n-1)!} \int_x^\infty
    (t-x)^{n-1}\:
     K_0(t) \:dt, n \ge 1.
\end{equation}

In our previous work \cite{Ruffa2024} we proved a fundamental identity for the first order Bickley-Naylor function $Ki_1,$ which we recall here as a lemma to our main theorem.

\begin{lemma}
\begin{equation}
Ki_1(x)
=\int_0^{\frac{\pi}{2}} e^{-x\sec\theta}\: d\theta\\
=\frac{\pi}{2} - \frac{\pi x}{2} \left[ K_0(x) \:\pmb{L}_{-1}(x) + K_1(x) \:\pmb{L}_0(x) \right], 
\label{Ki1}
\end{equation}
\end{lemma}
where, for $x\in \mathbb{R}^+,$ $K_0(x)$ and $K_1(x)$ are modified Bessel functions, and $\pmb{L}_{-1}(x)$ and $\pmb{L}_{0}(x)$ are modified Struve functions.

\begin{remark}
    Importantly, we rewrite the above identity as follows:
\begin{equation}
Ki_1(x)
=
\frac{\pi}{2} -\frac{\pi x}{2} \:A(x), x \in \mathbb{R}^+,
\label{Ki1!}
\end{equation}
where $A(x) = K_0(x) \:\pmb{L}_{-1}(x) + K_1(x)\:\pmb{L}_0(x)$, i.e., a combined modified Bessel-Struve function. 

In other words, $Ki_1(x) = C(x) + P(x) \:A(x)$ is a real polynomial combination of $\{1,A\}$ for $C(x)=\frac{\pi}{2}$ and $P(x) = -\frac{\pi x}{2}.$
This identity provides an exact analytic expression for the first order Bickley-Naylor function in terms of the modified Bessel functions $K_0$ and $K_1$ and the modified Struve functions $\pmb{L}_0$ and $\pmb{L}_{-1}.$ It is one of the strongest explicit closed-form representations in the sequence of Bickley-Naylor sequences, leading it to be treated as a separate special function in its own right. 
\end{remark}

\section{A System of Equations}

Our approach is described in the following. First, we use \emph{Mathematica} to evaluate a multidimensional integral that is an extension of the types of integrals that appeared in \cite{Ruffa2024}. We state the result also as a lemma to our main theorem

\begin{lemma}

\begin{equation}
I
=
\int_0^\infty
\int_0^\infty
\int_0^\infty
e^{
-\frac{x^2}{4}
(x^2+y+z)
}\cdot
e^{-(y+z)^{-1}}\:
dx\:dy\:dz
=
\frac{8\sqrt{\pi}}{x^3}\:
K_2(x).
\label{Integrals4340a}
\end{equation}

\end{lemma}

We then develop a $3\times 3$ system of equations involving Bickley-Naylor functions of consecutive orders 2, 3, and 4. We first prove

\begin{theorem}

\begin{equation}
x^2\:Ki_2(x) + 3x\:Ki_3(x) + 3\:Ki_4(x) = x^2\:K_2(x).
\label{Integrals4340}
\end{equation}
    
\end{theorem}

\begin{proof}
Consider the integral identity in \eqref{Integrals4340a}, i.e.,

\begin{equation}
I
=
\int_0^\infty
\int_0^\infty
\int_0^\infty
e^{
-\frac{x^2}{4}
(x^2+y+z)
}\cdot
e^{-(y+z)^{-1}}\:
dx\:dy\:dz
=
\frac{8\sqrt{\pi}}{x^3}\:
K_2(x).
\label{Integrals4340aa}
\end{equation}

By design, this integral is amenable to the multivariate power substitution approach described and extensively illustrated in \cite{ Ruffa2022,Ruffa2024}. Specifically, we apply the substitutions $x=\sqrt{y+z}\:u_1$, $dx=\sqrt{y+z}\: du_1$, $y=z\:u_2$, and $dy=z\: du_2$ and obtain

\begin{equation}
\begin{split}
I = & \int_0^\infty \int_0^\infty \int_0^\infty e^{-\frac{1}{4}\:z\:x^2(1+u_1^2)(1+u_2)}\cdot
e^{-z^{-1}\:(1+u_2)^{-1} } z^{\frac{3}{2}}\: \sqrt{1+u_2}\: dz\:du_2\:du_1\\
= & \frac{8\sqrt{\pi}}{x^5}
\int_0^\infty
e^{-x\sqrt{1+u_1^2}}\cdot
\frac{ 3+x^2(1+u_1^2)+3\:x\:\sqrt{1+u_1^2}}{\sqrt{(1+u_1^2)^5}}\:du_1.
\label{Integrals4340b}
\end{split}
\end{equation}

The expression \eqref{Integrals4340b} clearly suggests a trigonometric substitution. We further apply the substitution $u_1=\tan\theta$ and $du_1=\sec^2\theta\:d\theta$, which leads to 

\begin{equation}
I
=
\frac{8\sqrt{\pi}}{x^5}
\int_0^\frac{\pi}{2}
e^{-x\sec\theta}\:
(x^2\cos\theta+3x\cos^2\theta+3\cos^3\theta)\:
d\theta.
\label{Integrals4340c}
\end{equation}

To conclude, we equate \eqref{Integrals4340aa} and \eqref{Integrals4340c} and simplify. Thus the claim. \eqref{Integrals4340}.
\end{proof}

The $3\times 3$ system of equations comprise \eqref{Integrals4340} and the recurrence relation in \eqref{RR} with $n=2$ and with $n=3$, i.e.,

\begin{equation}
\begin{split}
& x^2\:Ki_2(x) + 3x\:Ki_3(x) + 3\:Ki_4(x) = x^2\:K_2(x);\\
& 2\:Ki_{3}(x) = -x\:Ki_2(x) + Ki_{1}(x) + x\:Ki_{0}(x) ;\\
& 3\:Ki_{4}(x) = -x\:Ki_3(x) + 2\:Ki_{2}(x) + x\:Ki_{1}(x).
\end{split}
\label{RR2}
\end{equation}

The unknowns are $Ki_2(x)$, $Ki_3(x)$, and $Ki_4(x)$. Substituting $Ki_0(x)=  K_0(x)$ and \eqref{Ki1} into \eqref{RR2} leads to

\begin{equation}
\begin{gathered}
\begin{bmatrix} 
x^2 & 3\:x & 3 \\
x & 2 & 0 \\ 
-2 & x & 3
\end{bmatrix} 
\begin{Bmatrix} 
Ki_2(x)\\
Ki_3(x)\\
Ki_4(x)
\end{Bmatrix} 
= 
\begin{Bmatrix} 
x^2\:K_2(x) \\ 
Ki_1(x)+x \:K_0(x) \\
x \:Ki_1(x)
\end{Bmatrix}\\
=
\begin{Bmatrix} 
x^2\:K_2(x) \\ 
\frac{\pi}{2} + x\:K_0(x) - \frac{\pi\:x}{2} A(x) \\
\frac{\pi\:x}{2}  - \frac{\pi\:x^2}{2} A(x)  
\end{Bmatrix}.
\end{gathered}
\label{syseqs}
\end{equation}

Solving the system of equations \eqref{syseqs} is straightforward (e.g., via \emph{Mathematica}) and it yields the corollary

\begin{corollary}

\begin{equation} 
\begin{split}
Ki_2(x)
= & -\frac{\pi x}{2} + x K_1(x) + \frac{\pi x^2}{2} [ K_0(x) \:\pmb{L}_{-1}(x) + K_1(x) \:\pmb{L}_0 (x)]\\
= & P_2(x)\:K_1(x) + Q_2(x)\: A(x) + C_2(x),
\end{split}
\end{equation}
where $P_2(x)=x;$ $Q_2(x) = \frac{\pi x^2}{2};$ $C_2(x)= -\frac{\pi x}{2},$

and

\begin{equation} 
\begin{split}
Ki_3(x)
= & \frac{\pi (1+x^2)}{4} + \frac{x}{2} K_0(x)
- \frac{x^2}{2} K_1(x)\\
& - \frac{\pi x (1+x^2)}{4} [ K_0(x) \:\pmb{L}_{-1} (x)
+  K_1(x) \:\pmb{L}_0(x)] \\
= & P_3(x)\:K_1(x) + R_3(x)\: K_0(x) + Q_3(x) \:A(x) + C_3(x),
\end{split}
\end{equation}

where $P_3(x)=-\frac{x^2}{2};$ $R_3(x)=\frac{x}{2};$ $Q_3(x) =- \frac{\pi x (1+x^2)}{2};$ $C_3(x)= \frac{\pi(1+x^2)}{4},$

and

\begin{equation} 
\begin{split}
Ki_4(x)
= & -\frac{\pi x (3+x^2)}{12} - \frac{x^2} {6}K_0(x) + \frac{x (4+x^2)}{6} K_1(x) \\
& + \frac{\pi x^2 (3+x^2)}{12} 
[K_0(x) \:\pmb{L}_{-1} (x) + K_1(x) \:\pmb{L}_0 (x) ] \\
= & P_4(x) K_1(x) + R_4(x)  K_0(x) + Q_4(x)  A(x) + C_4(x),
\end{split}
\end{equation}
where $P_4(x)=\frac{\pi(4+x^2)}{6};$ $Q_4(x) = \frac{\pi x^2(3+x^2)}{12};$ $R_4(x)= -\frac{x^2}{6};$ $ C_4(x)=-\frac{\pi x (3+x^2)}{12}.$

\end{corollary}

To further extend the reach of this result, we express the first four terms of the Bickley-Naylor sequence $Ki_n: \forall n\in\mathbb{N}$, i.e., $Ki_1,$ $Ki_2, Ki_3,$ and $Ki_4$ in terms of modified Bessel functions $K_0(x),$ $K_1(x)$ and the combined modified Bessel-Struve function $A(x)$ with real polynomial coefficients. This is better stated algebraically as follows:

The Bickley-Naylor functions $Ki_{n=1,2,3,4}$ are real polynomial combinations of $\{1, K_0, K_1, A\}$ with polynomial coefficients. We denote: $\forall n=1,2,3, 4$ 

\begin{equation}
\begin{split}
Ki_n \in \mathbb{M} = & \langle 1,K_0, K_1, A \rangle _{\mathbb{R}[x]}\\
 = & \operatorname{Span}_{\mathbb{R}[x]}(\{1,K_0,K_1,A\}),  
\end{split}
\end{equation}
the module $\mathbb{M}$ finitely generated by the modified Bessel-Struve functions over the ring $\mathbb{R}[x]$ of real polynomials in $x.$

It remains to prove that the infinite family of the Bickley-Naylor functions, denoted $\mathbb{BN}=
\{Ki_n:n\in\mathbb{N}_0 \}$ is indeed  a subset of this four-generator module. 

A straightforward complete induction can be derived using the recurrence relation \eqref{RR}, to state

\begin{theorem}[Finite-Generation Theorem]

\begin{equation}
\{Ki_n|n\ge 0 \} \subset  \langle 1,K_0, K_1, A\rangle_{\mathbb{R}[x]}=\operatorname{Span}_{\mathbb{R}[x]}(\{1,K_0,K_1,A\}).
\end{equation}
\end{theorem}

\begin{proof} \noindent
\begin{enumerate}
\item $Ki_0 = K_0 \in \mathbb{M}, Ki_1 \in \mathbb{M}.$
\item Induction Hypothesis: Assume $ Ki_{n-2}, Ki_{n-1}, Ki_n \in \mathbb{M}.$ From the recurrence relation $nKi_{n+1}=-xKi_n+(n-1)Ki_{n-1}+Ki_{n-2}$, and the fact that $\mathbb{M}$ is a $\mathbb{R}[x]-$module, we derive $xKi_n, Ki_{n-1}, xKi_{n-2} \in \mathbb{M}.$ Therefore $Ki_{n+1} \in \mathbb{M}.$
\item In conclusion, $Ki_n\in\mathbb{M}, \forall n \ge 0 .$
\end{enumerate}
Hence, the claim.
\end{proof}

Note that our claim does not establish an equality between the infinite set of Bickley-Naylor functions or the span of the Bickley functions over $\mathbb{R}[x]$ and the four-generator $\mathbb{R}[x]-$module $\mathbb{M}$. Indeed, the constant function $1$ is not a Bickley function, as these functions actually  exponentially decay as $x \rightarrow \infty.$
Nonetheless, this result is indeed profound. 
 
\section{Why the Finitely Generated Module Structure Matters}

We reveal with the Finite-Generation Theorem that, beyond the explicit representation of $ Ki_2,$ $Ki_3$, $Ki_4$, the family of the Bickley-Naylor functions possesses a structure of a  $\mathbb{R}[x]-$module generated by the four functions $\{1, K_0, K_1, A\},$ with $A(x)=K_0 \:\pmb{L}_{-1} + K_1 \:\pmb{L}_0$. This is indeed a new structural perspective on the Bickley-Naylor functions which has several implications.

\begin{enumerate}
    \item Finite Generation and Dimension Reduction: The Bickley-Naylor family of functions has traditionally been treated as an unstructured collection of independent ``special functions'' to be tabulated, approximated, or evaluated numerically. It is now revealed as a four-generator $\mathbb{R}[x]-$module. That is, every $Ki_n, n\ge 0$ can now be evaluated through only four modified Bessel and Struve functions involving some polynomial arithmetic.
    \item Computational consequences: The Finite-Generation Theorem reduces the computation and evaluation of a Bickley-Naylor function of any order to the computation of the generators $K_0,$ $K_1,$ $A$ with some polynomial operations. This is indeed an alternative computational framework to the traditional use of tabulation, recursive processes, or numerical approximation.
    
    \item Analytical and modeling implications: Every element $Ki_n, n\ge 0$ of the $\mathbb{BN}$ is also an element of the module generated by $\{1,K_0,K_1,A\}.$ Therefore it inherits any known feature and property of the modified Bessel-Struve generators. For example, the well-known asymptotics of $K_0,K_1, \pmb{L}_{-1}, \pmb{L}_0$ infer asymptotic expansions for every $Ki_n,n\in\mathbb{N}_0.$ They certainly lead to closed-form evaluations, new recurrence forms, and integral transforms in the applied sciences. This could prove useful in applications involving radiative heat transfer, neutron transport, and related kernel-based models.
\end{enumerate}

Therefore, the Finite-Generation Theorem uncovers an algebraic structure underlying the Bickley–Naylor family of functions. Such a structure would suggest new directions for both theoretical and computational investigations. For example, a finitely generated $\mathbb{R}[x]-$module structure may direct future research to the use of Gr\"obner-basis methods for algorithmic generation and simplification of the Bickley-Naylor functions.
 
\section{Concluding Remarks}
We studied the Bickley-Naylor family of functions $Ki_n,n\ge 0,$ using a non-standard approach based on the multivariate power substitution. We then derived and solved a novel $3\times 3$ system of equations to obtain explicit representations of $Ki_2,$ $Ki_3,$ and $Ki_4.$ These representations were naturally extended through the classical recurrence relation to higher-order Bickley-Naylor functions.

More significantly, we showed that every Bickley-Naylor function of the $\mathbb{BN}$ family is in fact an element of the finitely generated $\mathbb{R}[x]-$module $\mathbb{M}=\langle 1, K_0,K_1,A\rangle_{\mathbb{R}[x]}, A(x)=K_0 \:\pmb{L}_{-1} + K_1 \:\pmb{L}_0.$ That is, $\mathbb{BN},$ seemingly a collection of analytically unrelated special functions, admits a finite-generation framework with only four generators, i.e., a new perspective with potential implications for computation, asymptotic analysis, integral transforms, and numerical evaluation. The properties of the generators $K_0,K_1,A$ can support the derivation of equivalent properties for a Bickley-Naylor function of any order. This could be considered as a form of algebraic reduction for problems in radiative heat transfer, neutron transport, and related applications. Future work may extend the finite-generation framework to other special functions.

\end{document}